\documentclass[11pt]{amsproc}
\usepackage{mathrsfs}
\usepackage{stmaryrd}
\usepackage{cases}
\usepackage{amssymb}
\usepackage{amsmath}
\usepackage{amsfonts}
\usepackage{graphicx}
\usepackage{amsmath,amstext,amsbsy,amssymb}
\usepackage[table]{xcolor}
\usepackage{multirow}
\usepackage{cases}
\usepackage{color}

\newtheorem{theorem}{Theorem}[section]
\newtheorem{lemma}[theorem]{Lemma}

\newtheorem{proposition}{Proposition}[section]
\newtheorem{corollary}[theorem]{Corollary}

\theoremstyle{definition}
\newtheorem{definition}[theorem]{Definition}

\theoremstyle{remark}
\newtheorem{remark}[theorem]{Remark}

\numberwithin{equation}{section} \errorcontextlines=0

\newcommand{\la}{\lambda}

\begin{document}

\title[On modules for double affine Lie algebras]
{On modules for double affine Lie algebras}
\author{Naihuan Jing}
\address{School of Mathematical Sciences,
South China University of Technology, Guangzhou, Guangdong 510640, China and
Department of Mathematics,
   North Carolina State University,
   Raleigh, NC 27695, USA}
\email{jing@math.ncsu.edu}
\author{Chunhua Wang*}
\address{School of Mathematical Sciences,
South China University of Technology, Guangzhou, Guangdong 510640, China}
\email{$tiankon_{-}g1987@163.com$}

\keywords{double affine Lie algebras, Verma
modules, irreducibility, Weyl modules}
\subjclass[2010]{Primary: 17B67; Secondary: 17B10, 17B65}
\thanks{*Corresponding author.}

\begin{abstract}
Imaginary Verma modules, parabolic imaginary Verma modules,
and Verma modules at level zero for double affine Lie algebras are constructed
using three different triangular decompositions. Their relations are investigated,
and several results are generalized from the affine Lie algebras.
In particular, imaginary highest weight modules, integrable modules, and irreducibility criterion
are also studied.
\end{abstract}
\maketitle
\section{Introduction}

Let $\hat{\mathfrak{g}}$ be the untwisted affine Lie algebra
associated to a complex finite dimensional simple Lie algebra ${\mathfrak{g}}$
with Cartan subalgebra $\mathfrak{h}$.
The highest weight irreducible $\hat{\mathfrak{g}}$-module $L(\lambda)$
of the highest weight $\lambda$ can be studied with the help of the Verma module
$V(\lambda)$, which is an induced module
of the one-dimensional module $\mathbb C1_{\lambda}$ of the Borel subalgebra
$\hat{\mathfrak b}=\hat{\mathfrak{h}}+\hat{\mathfrak{n}}_+$
such that $\hat{\mathfrak{n}}_+1_\lambda=0$.
If one partitions the affine root system using the loop realization
 of $\hat{\mathfrak{g}}$, the associated imaginary Verma module \cite{F}
behaves quite differently. For example, the new Verma module can have infinite
dimensional weight subspaces.

Double affine Lie algebras are certain central extensions of maps
from a 2-dimensional torus to the Lie algebra ${\mathfrak{g}}$. They
are analogous to $\hat{\mathfrak g}$ but with two centers, and first
appeared in I.~Frenkel's work \cite{Fr} on affinization of Kac-Moody
algebras. Moody and Shi \cite{MS} have shown that the root systems
have different properties from those of the affine root systems. For
example, some roots can not be spanned positively or negatively by
the ``simple'' ones. Thus a usual highest weight module would blow
up beyond control. In this paper, we generalize imaginary Verma
modules (IVM) to double affine Lie algebras and use them to produce
irreducible modules.

To study these imaginary Verma modules, we consider
generalized imaginary Verma modules, which bare some similarity to
the parabolic Verma modules in classical Lie theory. The
structure of IVM's is characterized by the generalized
IVM's, and we show that they are irreducible if the second center is
nonzero. When the second center is zero, they are similar to
the evaluation modules of the affine Lie algebras, for which we also generalize several
results from the Tits-Kantor-Koecher algebras \cite{CT}.

To further understand the situation of trivial centers, we adopt
Chari and Pressley's technique \cite{CP, VT} of Weyl modules to
modules of double affine Lie algebras, which has played an important
role for loop algebras (see \cite{C} for a survey). We remark that
the triangular decomposition employed in our case is different from
that used in previous work on toroidal Lie algebras
(cf. \cite{FK, SE2,SE1}). In our modules, the Borel subalgebras are
defined by carving out the second imaginary part to control the
growth of the IVM's.

\section{Double affine Lie algebras}

Let $\mathfrak{g}$ be a complex finite dimensional simple Lie
algebra of simply laced type with $\mathfrak{h}$ its Cartan
subalgebra. Let $\Delta$ be the root system generated by the simple
roots $\alpha_i~(i=1, \ldots, s)$ and let $\alpha_i^{\vee}\in
\mathfrak h$ ($i=1, \ldots, s$) be the corresponding simple coroots
such that $\langle \alpha_j^{\vee}, \alpha_i\rangle=a_{ij}$, the
entries of the Cartan matrix of $\mathfrak g$. Denote by
$\theta=\sum\limits_{i=1}^s k_i\alpha_i$ the longest root of
$\mathfrak{g}$, and $\theta^{\vee}$ the corresponding dual element
in $\mathfrak{h}$. For any positive root
$\alpha=\sum_ic_i\alpha_i\in\Delta$, we denote its height by
$ht(\alpha)=\sum_{i=1}^sc_i$.

The double affine Lie algebra $\overline{\mathfrak{T}}$ is the
central extension of the 2-loop algebra defined by
$$\overline{\mathfrak{T}}= \mathfrak{g}\otimes
\mathbb{C}[t_1,t_1^{-1},t_2,t_2^{-1}]\oplus \mathbb{C}c_1\oplus
\mathbb{C}c_2,$$
where $\mathbb{C}[t_1,t_1^{-1},t_2,t_2^{-1}]$ is
the ring of Laurent polynomials in two commuting variables $t_1$
and $t_2$. Writing $x\otimes t_1^m t_2^n$ as $x(m,n)$, the Lie
bracket is given by
\begin{align}
[x(m_1, n_1), y(m_2, n_2)]&=[x, y](m_1+m_2, n_1+n_2)\\
  \nonumber                            &+(x_1|x_2)\delta_{m_1,-n_1}\delta_{m_2,-n_2}(m_1c_1+m_2c_2),\\
[c_1, x(m_1, n_1)]&=[c_2, x(m_1, n_1)]=0,
\end{align}
where $(m_1,m_2), (n_1,n_2)\in\mathbb{Z}^2$, $x_1,x_2,x\in
\mathfrak{g}$, and $(x_1|x_2)$ is the $\mathfrak{g}$-invariant
bilinear form. Adjoining the derivations, we define
the extended double affine Lie algebra ${\mathfrak T}$ as
$$\mathfrak T= \mathfrak{g}\otimes
\mathbb{C}[t_1,t_1^{-1},t_2,t_2^{-1}]\oplus \mathbb{C}c_1\oplus
\mathbb{C}c_2\oplus \mathbb{C}d_1\oplus \mathbb{C}d_2.$$
The derivations act on $\mathfrak T$ via
\begin{equation}
[d_i,x(m_1,m_2)]= m_ix(m_1, m_2), \quad [d_i, c_j]=0, \quad
i,j=1,2.
\end{equation}

Let $\hat{\mathfrak h}=\mathfrak{h}\oplus\mathbb Cc_1\oplus\mathbb
Cc_2\oplus\mathbb Cd_1 \oplus\mathbb Cd_2$ be the Cartan subalgebra
of $\mathfrak T$, and $\hat{\mathfrak h}^{*}$ be the dual space.
For $\beta\in\hat{\mathfrak h}^{*}$ the root subspace $\mathfrak
T_{\beta}=\{x\in\mathfrak T|[h,x]=\beta(h)x, \, \forall \,
h\in\hat{\mathfrak h}\}$ is defined in the usual way. We then define
the root system $\Delta_{\mathfrak T}$ to be the set of all nonzero
$\beta\in\hat{\mathfrak h}^{*}$ such that $\mathfrak T_{\beta}\neq
0$. It is
known that \cite{MS} the root system is different from
the usual root system in that a root is no longer a positive or
negative sum of the ``simple'' roots.

To be specific, we let $\delta_i\in\hat{\mathfrak
h}^{\ast}$ such that $\delta_i(\mathfrak{h})=\delta_i(c_j)=0$, and
$\delta_i(d_j)=\delta_{ij}$ for $i,j=1,2$. Then the extended root
system
$$\Delta_{\mathfrak T}=\{\alpha+\mathbb{Z}\delta_1+\mathbb{Z}\delta_2|\alpha\in\Delta\}
\cup(\{\mathbb{Z}\delta_1+\mathbb{Z}\delta_2\}\setminus\{0\}),$$
where the first and second subsets form the real and imaginary roots, denoted
by $\Delta_{\mathfrak T}^{re}$ and $\Delta_{\mathfrak
T}^{im}$ respectively. The corresponding root subspaces are
\begin{align*}
\mathfrak T_{m\delta_1+n\delta_2}=&\mathfrak{h}\otimes
t_1^mt_2^n,\quad (m,n)\neq(0,0), \\
\mathfrak
T_{\alpha+m\delta_1+n\delta_2}=&\mathfrak{g}_{\alpha}\otimes
t_1^mt_2^n,\quad (m,n)\in \mathbb{Z}\times\mathbb{Z}.
\end{align*}
Then one has the root space decomposition
$$\mathfrak T=
\hat{\mathfrak h}\oplus\bigoplus_{\beta\in\Delta_{\mathfrak T}}\mathfrak T_{\beta}.$$

Obviously the extended double affine Lie algebra $\mathfrak T$
contains two affine Lie algebras as subalgebras:
$$\hat{\mathfrak{g}}_i=\mathfrak{g}\otimes\mathbb{C}[t_i,t_i^{-1}]\oplus\mathbb{C}c_i\oplus\mathbb{C}d_i,
\quad i=1, 2.$$
Recall that \cite{VG} $\alpha_{0}=\delta_1-\theta,\alpha_1,\ldots,\alpha_s$
are the simple roots of $\hat{\mathfrak g}_1$. Similarly, the roots
$\alpha_{-1}=\delta_2-\theta,\alpha_1,\ldots,\alpha_s$ are simple roots for
$\hat{\mathfrak g}_2$. Following \cite{MS} we call the elements
$\alpha_{-1},\alpha_0,\alpha_1,\ldots,\alpha_s$ the
``simple" roots of $\mathfrak T$. However, not all
roots can be represented as non-negative or non-positive linear
combination of the simple roots. The corresponding coroots are
$\alpha_{-1}^{\vee},\alpha_0^{\vee},\alpha_1^{\vee},\ldots,\alpha_s^{\vee}$,
where $\alpha_{-1}^{\vee}=c_2-\theta^{\vee}\otimes 1$ and
$\alpha_0^{\vee}=c_1-\theta^{\vee}\otimes 1$. Consequently the
Cartan subalgebra $\hat{\mathfrak h}$ is spanned by
$\alpha_{-1}^{\vee},\alpha_0^{\vee},\alpha_1^{\vee},\ldots,\alpha_s^{\vee},d_1,d_2$.

\section{Imaginary Verma modules of $\mathfrak T$}
Verma modules of affine Lie algebras are defined with the help of a
triangular decomposition, which is constructed by a choice of
the base in the root system. Since the root system for the extended
double affine algebras can not be divided into positive and
negative roots in the usual sense, we use a closed subset to
partition the root system. For affine Lie algebras Futorny \cite{F}
studied a new class of Verma modules called the imaginary Verma
modules (IVM) \cite{DFF, RF} associated with a
closed subset defined by a function. For finite dimensional simple Lie
algebras, the partition derived from the closed subset coincides
with the usual partition of positive and negative roots. In this
section, we introduce the imaginary Verma modules for the extended
double affine Lie algebras and derive their important properties.

\subsection{Imaginary Verma modules}

Set $\varphi=\sum\limits_{i=0}^s
\alpha_{i}^{\ast}-ht(\theta)\alpha_{-1}^{\ast}$, where
$\alpha_i^{\ast}\in\hat{\mathfrak h}$ such that
$\alpha_i^{\ast}(\alpha_j)=\delta_{ij} ~(i,j=-1,0,\ldots,s)$.
Clearly, $\varphi(\alpha_i)=1$ for $i=0,1,\ldots, s$,
$\varphi(\delta_2)=0$, and $\varphi(\delta_1)=1+ht(\theta)$.  Let
$\mathcal {I}=\{\alpha\in\Delta_{\mathfrak T}|\varphi(\alpha)>
0\}\cup \mathbb{N}\delta_2$.
Clearly it is closed under the addition.

Recall \cite{VG} that the root system of $\hat{\mathfrak g}_1$ is
$\Delta_{\hat{\mathfrak
g}_1}=\{\alpha+\mathbb{Z}\delta_1|\alpha\in\Delta\}\cup(\{\mathbb{Z}\delta_1\}\backslash\{0\})$,
and its positive roots are $\Delta_{\hat{\mathfrak
g}_1+}=\{\pm\alpha+\mathbb{N}\delta_1|\alpha\in\Delta\}\cup\Delta_+\cup
\mathbb N\delta_1$. Note that for any
$\alpha\in\Delta_{\hat{\mathfrak g}_1+}$,
$\varphi(\alpha)=ht(\alpha)$. Then $\mathcal {I}$ can be written as
$\mathcal
{I}=\{\alpha+\mathbb{Z}\delta_2|\alpha\in\Delta_{\hat{\mathfrak
g}_1+}\}\cup\mathbb{N}\delta_2$. Subsequently
$\mathcal {I}\cup(-\mathcal {I})=\Delta_{\mathfrak T}$ and $\mathcal
{I}\cap(-\mathcal {I})=\emptyset$. Denote by $\mathcal Q(\mathcal
I)$ the $\mathbb Z$-span of $\mathcal I$.

Let $\mathfrak T_{\mathcal {I}}=\oplus_{\beta\in\mathcal {I}}
\mathfrak T_{\beta}$ and $\mathfrak
T_{-\mathcal{I}}=\oplus_{\beta\in-\mathcal {I}} \mathfrak
T_{\beta}$, then
\begin{equation}\label{e:triangle}
\mathfrak T=\mathfrak T_{-\mathcal {I}}\oplus
\hat{\mathfrak h}\oplus\mathfrak T_{\mathcal {I}}
\end{equation}
 is a triangular
decomposition of $\mathfrak T$ associated with the closed subset
$\mathcal {I}$. The Poincar\'e-Birkhoff-Witt theorem implies that the
universal enveloping algebra
$$U(\mathfrak T)=U(\mathfrak T_{-\mathcal {I}})\otimes
U(\hat{\mathfrak h})\otimes U(\mathfrak T_{\mathcal {I}}).
$$

Let ${\mathfrak b}_{\mathcal
{I}}=\hat{\mathfrak h}\oplus\mathfrak T_{\mathcal {I}}$ be
the imaginary Borel subalgebra, which is a solvable Lie algebra.
A vector space $V$ is called a weight module of $\mathfrak T$ if
$$V=\bigoplus_{\mu\in\hat{\mathfrak h}^*}V_{\mu},$$
where $V_{\mu}=\{v\in V|hv=\mu(h)v, \, \forall \, h\in\hat{\mathfrak
h}\}$.  Set $\mathcal {P}(V)=\{\mu\in\hat{\mathfrak
h}^*|V_{\mu}\neq 0\}$. We say that
$\lambda\geq\mu$ ($\lambda,\mu\in\hat{\mathfrak h}^*$)
with respect to $\varphi$ if $\lambda-\mu$ is a
non-negative linear combination of roots in
${\mathcal I}$. For simplicity we will omit the reference to
$\varphi$ if no confusion arises from the context. For example,
both $\mathfrak T$ and $U(\mathfrak T)$ are weight modules for
$\mathfrak T$.

\begin{definition}
Let $\lambda\in\hat{\mathfrak h}^{\ast}$. A nonzero vector $v$ is
called an imaginary highest vector with weight $\lambda$ if
$\mathfrak T_{\mathcal {I}}.v=0$ and $h.v=\lambda(h)v$ for all
$h\in\hat{\mathfrak h}$. If $V=U(\mathfrak T)v$ then $V$ is called a
highest weight module of highest weight $\lambda$.
\end{definition}

If $V$ is a highest weight module of weight $\la$, then
\begin{equation}
V=U(\mathfrak T_{-\mathcal I})v
=\bigoplus_{\eta\in\mathcal Q(\mathcal I)^{+}}V_{\lambda-\eta},
\end{equation}
where $\mathcal Q(\mathcal I)^+=\mathbb Z_+$-span of $\mathcal I$.

For $\lambda\in\hat{\mathfrak h}^{\ast}$, let $\mathbb C1_{\lambda}$
be the one-dimensional $\mathfrak b_{\mathcal {I}}-$module defined
by $(x+h).1_{\lambda}=\lambda(h)\cdot 1_{\lambda}$ ($x\in \mathfrak
T_{\mathcal {I}}$, $h\in\hat{\mathfrak h}$). The imaginary Verma
module is the induced module
$$\overline{M}(\lambda)=U(\mathfrak T)\otimes_{U(\mathfrak b_{\mathcal
{I}})}\mathbb C1_{\lambda}.$$

\subsection{Properties of IVM}
Based on the theory of the standard Verma modules \cite{VG} and IVM's
for the affine Lie algebras \cite{F}, the following result can be
proved similarly.

\begin{proposition}\label{P:imv} For any $\lambda\in\hat{\mathfrak h}^*$, one has that
\begin{enumerate}
\item $\overline{M}(\lambda)$ is a $U(\mathfrak T_{-\mathcal
{I}})-$free module generated by the imaginary highest vector
$1\otimes 1_{\lambda}$ of weight $\lambda$.

\item \begin{enumerate}

\item $\dim\overline{M}(\lambda)_{\lambda}=1$;

\item $0<\dim\overline{M}(\lambda)_{\lambda-k\delta_2}<\infty$ for every
positive integer $k$;

\item If $\overline{M}(\lambda)_{\mu}\neq 0$ and
$\mu\neq \lambda-k\delta_2$ for any nonnegative integer $k$, then we
have $\dim\overline{M}(\lambda)_{\mu}=\infty$.
\end{enumerate}

\item  Any imaginary highest weight $\mathfrak T-$module
of highest weight $\lambda$ is a
quotient of $\overline{M}(\lambda)$.

\item  $\overline{M}(\lambda)$ has a unique maximal submodule $\mathcal {J}$.

\item If $\mu\in\hat{\mathfrak h}^{\ast}$, then any nonzero
homomorphism $\overline{M}(\lambda)\rightarrow\overline{M}(\mu)$
is injective.
\end{enumerate}
\end{proposition}

We denote by $\overline{L}(\lambda)$ the irreducible quotient
$\overline{M}(\lambda)/\mathcal {J}$.

\subsection{Irreducibility criterion for IVM}
Futorny \cite{F} found that the affine imaginary Verma module is irreducible
if and only if the center acts nontrivially. It turns out that
a similar result can be obtained for $\overline{M}(\lambda)$
(see Theorem \ref{T:crierion}).

\begin{lemma}\label{1}
Let $\overline{M}=\bigoplus_{j=0}^{\infty}\overline{M}(\la)_{\la-j\delta_2}$.
For any nonzero $v\in \overline{M}(\lambda)$,
$U(\mathfrak T)v\cap \overline{M}\neq 0$.
\end{lemma}
\begin{proof} Write $\lambda=\overline{\lambda}-r\delta_2$, where $\overline{\lambda}$
is the component of $\Delta_{\hat{\mathfrak g}_1}$. We can assume that
$r$ is minimum so that $v_{\la}\neq 0$, otherwise we can replace
$\lambda$ by $\lambda'$ and $\overline{M}(\lambda)=\overline{M}(\lambda')$ where
$\lambda\equiv\lambda'\, mod(\mathbb Z\delta_2)$. Therefore $ht_2^{-n}v_{\la}\neq 0$
for any $n\geq 0$.
Assume that $v\in \overline{M}(\lambda)_{\lambda-\mu}$,
where $\mu=\sum\limits_{i=0}^s n_{i}\alpha_{i}+k\delta_2$,
$n_i\in\mathbb{Z}_+$, and $k\in \mathbb{Z}$. Define the height $ht(\mu)= \sum\limits_{i=0}^s n_i$ and use
induction on $ht(\mu)$. If
$ht(\mu)=0$, the weight of $v$ is $\lambda-k\delta_2$ for some
$k\in\mathbb Z_+$ by assumption, so the result clearly holds.

Let $e_i, f_i, \alpha_i^{\vee} \, (0\leq i\leq s)$ be the Chevalley
generators of the derived affine Lie algebra $\hat{\mathfrak g}_1$.
If $ht(\mu)>0$, there exists $i_0\in\{0, \cdots, s\}$ such that
$e_{i_0}v\neq 0$. In fact, if $ht(\mu)=1$, say
$v=f_it_2^{-k}v_{\la}$. If $k<0$, then
$\alpha_i^{\vee}t_2^{k-1}v=-2f_it_2^{-1}v_{\la}\neq 0$. Then let
$v'=\alpha_i^{\vee}t_2^{k-1}v$ and
$e_iv'=-2\alpha_i^{\vee}t_2^{-1}v_{\la}\neq0$. The
case of $ht(\mu)\geq 2$ is treated similarly. Moreover,
$e_{i_0}(h\otimes t_2^{-m}).v'\neq 0$ for all $m\geq 0$. Hence
$\mathfrak T_{\alpha_{i_0}-m\delta_2}.v'\neq 0$, and any of its
nonzero element has weight $\la-\mu+\alpha_{i_0}-(m+1)\delta_2$. As
$ht(\mu-\alpha_{i_0})=ht(\mu)-1$, by inductive hypothesis we have
$U(\mathfrak T)(\mathfrak T_{\alpha_{i_0}-m\delta_2}.v')\cap
\overline{M}\neq 0$. Since $U(\mathfrak T)(\mathfrak
T_{\alpha_{i_0}-m\delta_2}.v') \subset U(\mathfrak T)v$, it follows
that $U(\mathfrak T)v\cap \overline{M}\neq 0$.\end{proof}

Let $\overline{M}(\lambda)^{+}=\{v\in\overline{M}(\lambda)|\mathfrak
T_{\mathcal {I}}.v=0\}$ be the space of extremal vectors. Clearly
$\overline{M}(\lambda)^{+}$ is $\hat{\mathfrak h}-$invariant. For
any nonzero element $v\in\overline{M}(\lambda)^{+}$, $U(\mathfrak
T).v$ generates a submodule of $\overline{M}(\lambda)$. The
following result describes the form of extremal vectors.

\begin{corollary}\label{subset}
$\overline{M}(\lambda)^{+}\subset \overline{M}$.
\end{corollary}
\begin{proof} Suppose on the contrary that there exists a nonzero $v\in \overline{M}(\lambda)^{+}\cap
\overline{M}(\lambda)_{\lambda-\mu}$ such that
$\mu=\sum\limits_{i=0}^s n_i\alpha_i+k\delta_2$ and $ht(\mu)=
\sum\limits_{i=0}^s n_i> 0$. Note that $U(\mathfrak T)v= U(\mathfrak
T_{-\mathcal {I}})v$, so the weight of any homogeneous vector in
$U(\mathfrak T)v$ is $\lambda-\mu-\nu$ for $\nu\in\mathcal
{Q}(\mathcal I)^{+}$. As $ht(\mu)>0$,
$\lambda-\mu-\nu\neq \lambda (mod\, \mathbb Z\delta_2)$. Hence
$U(\mathfrak T)v\cap \overline{M}= 0$, which contradicts with Lemma
\ref{1} on
$\overline{M}(\lambda)_{\lambda-\mu}$.
\end{proof}

Define the Heisenberg subalgebra $\hat{\mathfrak
H}_2=\bigoplus_{n\in\mathbb Z^{\times}}\mathfrak h\otimes
t_2^n+\mathbb Cc_2$, then the space
$\overline{M}=\bigoplus_{j=0}^{\infty}\overline{M}(\la)_{\la-j\delta_2}$
is a Verma module for 
$\hat{\mathfrak H}_2$.
The following result is well-known from Stone-Von Neumann's theorem.

\begin{lemma}\label {2}
$\overline{M}$ is irreducible as a Verma module for $\hat{\mathfrak H}_2$ if
and only if $\lambda(c_2)\neq 0$.
\end{lemma}

\begin{theorem}\label{T:crierion}
The imaginary Verma module $\overline{M}(\lambda)$ is irreducible if
and only if $\lambda(c_2)\neq 0$.
\end{theorem}
\begin{proof} Let $v_{\lambda}$ be the highest weight vector
of $\overline{M}(\la)$. By definition $\overline{M}(\lambda)$ is irreducible if and only if
the space of extremal vectors $\overline{M}(\lambda)^+=\mathbb C v_{\lambda}$.

Suppose that $\overline{M}(\lambda)$ is irreducible but
$\lambda(c_2)= 0$. Since $\overline{M}$ is reducible as an
$\hat{\mathfrak{H}}_2$-module by Lemma \ref {2}, there exists $w\neq
0$ with weight $\lambda-k\delta_2$ ($k>0$) such that $\mathfrak
T_{l\delta_2}.w=0$ for any $l>0$. It is clear that $\mathfrak
T_{\beta}.w=0$ for all $\beta\in \{\alpha\in\Delta_{\mathfrak
T}|\varphi(\alpha)> 0\}$, because the weight of $\mathfrak
T_{\beta}.w$ is larger than $\la$. Thus
$w\notin \mathbb Cv_{\lambda}$ is also an imaginary highest vector
in $\overline{M}(\lambda)$, which is a contradiction.
 So we must have $\lambda(c_2)\neq 0$.

If $\lambda(c_2)\neq 0$, Lemma \ref {2} implies that the
$\hat{\mathfrak{H}}_2-$module $\overline{M}$ is irreducible.
Consider any nonzero submodule $U(\mathfrak T)v$, where $v\in
\overline{M}(\lambda)$. Lemma \ref {1} says that $U(\mathfrak
T)v\cap \overline{M}\neq 0$. Note that $U(\mathfrak T)v\cap
\overline{M}$ is then a non-trivial $\hat{\mathfrak{H}}_2$-submodule
of $\overline{M}$, consequently $U(\mathfrak T)v\cap
\overline{M}=\overline{M}$, thus $\overline{M}\subset U(\mathfrak
T)v$. Because $v_{\la}\in\overline{M}$, we have that $U(\mathfrak
T)v=\overline{M}(\lambda)$, i.e., $\overline{M}(\lambda)$ is
irreducible.\end{proof}

\section{Generalized IVM $M(\lambda,\mathcal {A})$ and highest weight modules}
In this section, we give a new class of modules $M(\lambda,\mathcal
{A})$ generalizing the previous IVM's to study the structure of IVM.
They are similar to parabolic Verma modules.

\subsection{Definition of $M(\lambda,\mathcal {A})$}
Let $\mathcal {A}\subset \mathcal {B}=\{0,1,2,\ldots,s\}$,
the index set of $\hat{\mathfrak g}_1$. We denote $\mathcal
{A}^{\ast}=\mathcal {A}\backslash \{0\}$ and $\mathcal
{B}^{\ast}=\mathcal {B}\backslash \{0\}$. Define $f_{\mathcal
{A}}=\sum\limits_{i\in \mathcal {B}\backslash \mathcal {A}}
\alpha_i^{\ast}-(\sum\limits_{i\in \mathcal {B}^{\ast}\backslash
\mathcal {A}^{\ast}}k_i)\alpha_{-1}^{\ast}$ for $\mathcal {A}\neq
\mathcal {B}$ and $f_{\mathcal {B}}=0$, then $f_{\mathcal
{A}}(\delta_2)=0$. Set $Q(\mathcal
{A})=\{\alpha\in\Delta_{\mathfrak T}|f_{\mathcal
{A}}(\alpha)\geq 0\}$ which is closed under addition. Then
\begin{align*}
Q(\mathcal {A})=\mathcal {I}\ \cup\{-\sum \limits_{i\in \mathcal
{A}}l_i\alpha_i+\mathbb{Z}\delta_2|l_i\geq 0, \prod_il_i\neq 0\}\cup\{-\mathbb{N}\delta_2\}.
\end{align*}
Note that $\mathcal
{I}\subsetneqq Q(\mathcal {A}).$

The following result is clear.
\begin{proposition}
$Q(\mathcal {A})\cap(-Q(\mathcal {A}))=\{\sum\limits_{i\in \mathcal
{A}}\mathbb{Z}\alpha_i+\mathbb{Z}\delta_2\}\cap\Delta_{\mathfrak T}$.
\end{proposition}

Recall that the Cartan matrix is given by $\langle \alpha_j^{\vee}, \alpha_i\rangle=a_{ij}$  for
$i=-1,0,1,\ldots,s$, $j=0,1,\ldots,s$. Let
$\hat{\mathfrak h}_{\mathcal {A}}\subset\hat{\mathfrak h}$ be the space spanned by
$\alpha_i^{\vee}$~( $i\in \mathcal {A}$). Consider the subspaces
$\mathfrak T_{Q(\mathcal {A})}$ and
$\mathfrak T_{-\overline{Q(\mathcal {A})}}$, where
$\overline{Q(\mathcal {A})}=Q(\mathcal {A})\backslash (-Q(\mathcal
{A}))$. Then $\mathfrak T$ decomposes itself as:
$$\mathfrak T=\mathfrak T_{-\overline{Q(\mathcal {A})}}\oplus\hat{\mathfrak h}\oplus \mathfrak T_{Q(\mathcal
{A})}.$$

Let $\lambda\in\hat{\mathfrak h}^{\ast}$ such that $\la(\hat{\mathfrak h}_{\mathcal
{A}}\oplus\mathbb{C}c_2)=0$. Let $\mathbb C1_{\lambda}$ be the one-dimensional
$\hat{\mathfrak h}\oplus \mathfrak T_{Q(\mathcal {A})}-$module such that
$(x+h).1_{\lambda}=\lambda(h)\cdot 1_{\lambda}~(x\in
\mathfrak T_{Q(\mathcal {A})}$, $h\in\hat{\mathfrak h})$. Define the
induced $\mathfrak T$-module associated with $\mathcal {I}$, $\mathcal
{A}$ and $\lambda$:
$$M(\lambda,\mathcal {A})=U(\mathfrak T)\otimes_{U(\hat{\mathfrak h}\oplus \mathfrak T_{Q(\mathcal {A})}
)}\mathbb C1_{\lambda}.$$

\subsection{Properties of $M(\lambda,\mathcal {A})$}
The following result is similar to Proposition \ref{P:imv}.
\begin{proposition}\label{3} For any $\lambda\in\hat{\mathfrak h}^{\ast}$ such that $\la(\hat{\mathfrak h}_{\mathcal
{A}}\oplus\mathbb{C}c_2)=0$, one has that
\begin{enumerate}

\item  $M(\lambda,\mathcal {A})$ is a
$U(\mathfrak T_{-\overline{Q(\mathcal {A})}})-$free module
generated by $1\otimes 1_{\lambda}$.

\item $\dim M(\lambda,\mathcal {A})_{\mu}=0, 1$ 
for $\mu=\lambda-\sum\limits_{i\in \mathcal {A}}
 k_i\alpha_i-\alpha_j+k\delta_2$, $k_i\in\mathbb{Z}_+$, $k\in\mathbb{Z}$, and $j\in \mathcal
{B}\backslash \mathcal {A}$. Otherwise,
$\dim{M}(\lambda,\mathcal {A})_{\mu}=\infty$.

\item The $\mathfrak T-$module $M(\lambda,\mathcal {A})$ has a
unique irreducible quotient $L(\lambda,\mathcal {A})$.

\item Let
$\mathcal A''\subset \mathcal {A}$.
Then there exists a chain of
surjective homomorphisms
$$\overline{M}(\lambda)\rightarrow M(\lambda,\mathcal A'')\rightarrow M(\lambda,\mathcal {A}).$$

\item Let $\lambda,\mu\in\hat{\mathfrak h}^{\ast}$. Then every nonzero map in
$\text{Hom}_{\mathfrak T}(M(\lambda,\mathcal
{A}),M(\mu,\mathcal {A}))$ is injective.

\item Let
$\mathcal {A}\subset
\mathcal {B}$.

Then the module $M(\lambda,\mathcal {A})$ is irreducible if and only
if $\lambda(\alpha_i^{\vee})\neq 0$ for any $i\in \mathcal
A'\setminus \mathcal {A}$, $\mathcal {A}\varsubsetneqq \mathcal A'$,
i.e. $\mathcal A$ is the maximal set such that
$\la(\alpha_i^{\vee})=0$.
\end{enumerate}
\end{proposition}

\begin{remark}
If $\mathcal {A}=\mathcal {B}$, then $M(\lambda,\mathcal
{A})=L(\lambda,\mathcal {A})$ is a trivial one-dimensional module.
\end{remark}

\begin{corollary}
Let $\lambda\in\hat{\mathfrak h}^{\ast}$, $\mathcal {A}\subset \mathcal
{B}$, and  $\hat{\mathfrak h}_{\mathcal {A}}\oplus\mathbb{C}c_2\subset
ker\lambda$. Also assume that $\lambda(\alpha_i^{\vee})\neq 0$ for
any $i\in \mathcal {A}^{'}\setminus \mathcal {A}$, $\mathcal
{A}\varsubsetneqq \mathcal {A}^{'}$. Then
$\overline{L}(\lambda)\cong M(\lambda,\mathcal
{A})=L(\lambda,\mathcal {A})\cong
 L(\lambda,\mathcal {A}^{''})$ for every $\mathcal {A}^{''}\subset\mathcal {A}$.
\end{corollary}

\begin{proof} Item (6) in Proposition \ref{3} implies $ M(\lambda,\mathcal
{A})=L(\lambda,\mathcal {A})$. Meanwhile, item (4) in Proposition
\ref{3} implies $L(\lambda,\mathcal {A})\cong \overline{L}(\lambda)$
and $L(\lambda,\mathcal {A}^{''})\cong L(\lambda,\mathcal {A})$ for
every $\mathcal {A}^{''}\subset\mathcal {A}$. This completes the
proof.
\end{proof}

\begin{corollary}\label{4}
Let $\lambda\in\hat{\mathfrak h}^{\ast}$ and $\lambda(c_2)=0$. If
$\lambda(\alpha_i^{\vee})\neq 0$ for any $i\in \mathcal {B}$, then
$\overline{M}(\lambda)^{+}=\overline{M}$.
\end{corollary}

\begin{proof}
By Corollary \ref{subset}, it suffices to
show $\overline{M}\subset\overline{M}(\lambda)^{+}$. If
$\lambda(\alpha_i^{\vee})\neq 0$ for any $i\in \mathcal {B}$, then
$\overline{L}(\lambda)\cong
M(\lambda,\emptyset)=L(\lambda,\emptyset)$. The result follows by
comparing the definition of $\overline{L}(\lambda)$ and that of
$M(\lambda,\emptyset)$. \end{proof}

The following result is a consequence of Corollary \ref{4}.
\begin{theorem}
Suppose that $\lambda\in\hat{\mathfrak h}^{\ast}$, $\lambda(c_2)=0$ and
$\lambda(\alpha_i^{\vee})\neq 0$ for any $i\in \mathcal {B}$. Then one has that
\begin{enumerate}
\item $\overline{M}(\lambda)$ has infinitely many proper submodules:
\begin{equation*}
\overline{M}(\lambda)\supset\overline{M}(\lambda-\delta_2)\supset\overline{M}(\lambda-2\delta_2)\supset\cdots
\end{equation*}
where $\dim\overline{M}(\lambda-k\delta_2)_{\lambda-k\delta_2}=
\dim \overline{M}(\lambda)_{\lambda-k\delta_2}=m_k$
are finite. Moreover
$\overline{L}(\lambda)=\overline{M}(\lambda)/\overline{M}(\lambda-\delta_2)$.

\item The root multiplicities $\dim
\overline{M}(\la)_{\lambda-k\delta_2}=m_k$ for all
$k\geq 0$, and $M(\lambda-k\delta_2,\emptyset)$ exhaust all
irreducible subquotients of $\overline{M}(\lambda)$.

\item $\dim\text{Hom}_{\mathfrak T}(\overline{M}(\lambda-k\delta_2),\overline{M}(\lambda))=m_k$
for any integer $k\geq 0$.
\end{enumerate}
\end{theorem}

One can describe general highest weight modules as follows.
\begin{corollary}
Let $V$ be a highest weight $\mathfrak T$-module of highest weight
$\lambda$. If $c_2$ acts trivially and $\lambda(\alpha_i^{\vee})\neq
0$ for $i=0, \cdots, s$, then $V\simeq
\overline{M}(\lambda)/\overline{M}(\lambda-k\delta_2)$ for some $k$.
\end{corollary}

\section{Highest weight modules of $\mathfrak T$}
In this section, we construct another class of highest weight
$\mathfrak T$-modules by slightly modifying the triangular
decomposition for IVM. Using the method of \cite{CT}, we generalize
some results of TKK modules to the highest weight $\mathfrak
T$-modules under the condition that $\la(c_1)=\la(c_2)=0$.

These centerless modules are introduced to understand our earlier
IVMs and parabolic IVMs. We remark that our construction differs
from \cite{SE2, SE1} in that the Cartan subalgebra is purely
generated by the imaginary root $\delta_2$.

Let $\Phi_+=\mathcal {I}\backslash {\mathbb{N}\delta_{2}}$,
$\Phi_-=-\Phi_+$, and $\Phi_0=\mathbb{Z}\delta_2$.
Correspondingly, the root spaces are
$\mathfrak T_0=\bigoplus_{\alpha\in\Phi_0}\mathfrak T_{\alpha}$
and $\mathfrak T_{\pm}=\bigoplus_{\alpha\in\Phi_{+}}
\mathfrak T_{\pm\alpha}$. Obviously,
$\mathfrak T=\mathfrak T_+\oplus\mathfrak T_-\oplus\mathfrak T_0$,
and $
\Delta_{\mathfrak T}=\Phi_+\cup\Phi_-\cup(\Phi_0\backslash\{0\})$.

We define a new module structure on
$\mathbb C1_{\lambda}$ such that
$h.1_{\lambda}=\lambda(h)\cdot1_{\lambda}~(h\in\mathfrak T_0)$,
and $\mathfrak T_{+}.1_{\lambda}=0$. Similarly, we define
the induced module $M(\lambda)$ of $\mathfrak T$:
$$ M(\lambda)=U(\mathfrak T)\otimes_{U(\mathfrak T_0\oplus\mathfrak T_{+})}\mathbb C1_{\lambda}.$$

The following result describes the relations among IVM's, generalized IVM's, and highest weight modules.

\begin{proposition} For any $\la\in\hat{\mathfrak h}^*$, one has that
\begin{enumerate}
\item  $M(\lambda)$ is a $U(\mathfrak T_-)-$free module generated
by $1\otimes1_{\lambda}:=v_{\la}$.

\item $M(\lambda)$ has a unique
irreducible quotient $L(\lambda)$.

\item Let $\lambda\in\hat{\mathfrak h}^{\ast}$, $\hat{\mathfrak h}_{\mathcal
{A}}\oplus\mathbb{C}c_2\subset \ker\lambda$, and $\mathcal
A''\subset \mathcal {A}$. Then there exists a chain of
surjective homomorphisms
$$\overline{M}(\lambda)\rightarrow M(\lambda)\rightarrow M(\lambda,\mathcal A'')\rightarrow M(\lambda,\mathcal {A}).$$

\item Let $\mu\in\hat{\mathfrak h}^{\ast}$. Then every nonzero
element of $Hom_{\mathfrak T}(M(\lambda),M(\mu))$ is injective.

\item Let $\mathcal {A}\subset
\mathcal {B}$, and $\hat{\mathfrak h}_{\mathcal
{A}}\oplus\mathbb{C}c_2\subset \ker\lambda$. If
$\lambda(\alpha_i^{\vee})\neq 0$ for any $i\in \mathcal
{A}^{'}\setminus \mathcal {A}$, $\mathcal {A}\varsubsetneqq\mathcal
{A}^{'}$, then $\overline{L}(\lambda)\cong L(\lambda)\cong
L(\lambda,\mathcal {A})$.
\end{enumerate}
\end{proposition}

\begin{remark}
\begin{enumerate}
\item If $\mathcal {A}=\emptyset$, then $M(\lambda)=M(\lambda,\mathcal
{A})$.
\item $L(0)$ is one-dimensional. \end{enumerate}
\end{remark}

\subsection{Irreducible module $L(\lambda)$} In this subsection,
we study integrability of $L(\la)$ and its weight subspaces.

\begin{definition}A $\mathfrak T$-module $M$ is called integrable if $M$
is a weight module and $x_{\alpha}(m,n)\in\mathfrak
T~(\alpha\in\Delta; m,n\in\mathbb{Z})$ are locally
nilpotent on every nonzero $v\in M$, i.e., there exists a positive integer
$N=N(\alpha,m,n)$ such that $x_{\alpha}(m,n)^{N}.v=0$.
\end{definition}

We write $x(\alpha,n)=x_{\alpha}\otimes t_2^n$ for
$\alpha\in\Delta_{\hat{\mathfrak g}_1}$. The following result is clear.

\begin{lemma}\label{L:gen1} $\mathfrak T_{\pm}$ is generated by
$\{x(\pm\alpha_i,n)|i=0,1,\ldots,s; n\in\mathbb{Z}\}$ respectively.

\end{lemma}

For an arbitrary Lie algebra $\mathfrak{g}$, we recall the following
results.
\begin{proposition}\cite{VG}\label {5}
Let $v_1,v_2,\ldots $ be a system of generators of a $\mathfrak{g}-$
module $V$, and suppose that each $x\in \mathfrak{g}$ is locally $ad$-nilpotent
on $\mathfrak{g}$ and $x^{N_i}(v_i)=0$ for some
positive integers $N_i~(i=1,2,\ldots)$. Then $x$ is locally
nilpotent on $V$.
\end{proposition}

\begin{proposition}\cite{MP}\label {6}
Let $\pi:\mathfrak{g}\rightarrow gl(V)$ be a representation of
$\mathfrak{g}$. If both $ad x$ and $\pi(x)$ are locally nilpotent
for any $x\in\mathfrak g$, then for any $y\in \mathfrak{g}$, \begin{equation}
\pi(exp(ad x)(y))=(exp\, \pi(x))\pi(y)(exp\, \pi(x))^{-1}.
\end{equation}
\end{proposition}

For a real root $\gamma\in
\Delta_{\hat{\mathfrak
g}_1}^{re}=\{\alpha+n\delta_1|\alpha\neq0,~n\in\mathbb{Z}\}$,
define the reflection $r_{\gamma}$ on $\hat{\mathfrak h}^{*}$ by
$$r_\gamma(\lambda)=\lambda-\lambda(\gamma^{\vee})\gamma,$$ where
$\gamma^{\vee}=\alpha^{\vee}+nc_1$ if
$\gamma=\alpha+n\delta_1$. Let $W_{a}$ be the affine Weyl group of
$\hat{\mathfrak g}_1$ generated by the reflections
$r_\gamma$, $\gamma\in\Delta_{\hat{\mathfrak
g}_1}^{re}$. Then $W_{a}$ is a Coxeter group.

\begin{lemma}\label{7}
Suppose that $x(-\alpha_i,n)(i=0,1,\ldots,s, n\in\mathbb{Z})$ are
locally nilpotent on the highest weight vector
$v_{\lambda}$ in $L(\lambda)$. Then all $x(m,n)~(m\in\mathbb{Z})$
are locally nilpotent on $L(\lambda)$.
\end{lemma}

\begin{proof}
Since $x(-\alpha_i,n)(i=0,1,\ldots,s)$ are locally nilpotent on $v_{\lambda}$, they act locally nilpotent on any
$x(m,n)$ via the adjoint representation. By Prop. \ref{5} the
elements $x(-\alpha_i,n)$ are locally nilpotent on $L(\lambda)$. So
$L(\lambda)$ is integrable for each of the $sl_2$-algebras:
$<x(\alpha_i,n), x(-\alpha_i,-n),\alpha_i^{\vee}>$ ,
$i=1,2,\ldots,s$ as well as
$<x(\alpha_0,n),x(-\alpha_0,-n),\alpha_{0}^{\vee}=-\theta^{\vee}>$
(Note that we assume that $c_1,c_2$ act as zero).

Suppose $\beta=\alpha+m\delta_1+n\delta_2\in
\Delta_{\mathfrak T}^{re}$ is the root of $x(m,n)$. Let $\gamma=\alpha+m\delta_1$, then
$\gamma=\beta-n\delta_2\in \Delta_{\hat{\mathfrak
g}_1}^{re}$. For any $i\in\{0, \ldots, s\}$, there exists a $w\in
W_a$ such that $w(\gamma)=\alpha_i$ \cite{VG}. Since
$w(\delta_2)=\delta_2$, we have $w(\beta)=\alpha_{i}+n\delta_2$. Let
$s_w$ be the linear automorphism of $\mathfrak T$ associated with
$w$. Up to a nonzero constant{\color{blue},} we have $s_w(x(m,n))=Y$
for $Y\in \mathfrak T_{\alpha_i+n\delta_2}$. It follows from Prop.
\ref{6} that all $x(m,n)$ are locally nilpotent on $L(\lambda)$.
\end{proof}

We now recall Weyl modules \cite{CP} for the loop algebra
$\hat{\mathfrak{sl}}_2(\mathbb{C})=\mathfrak{sl_2}\otimes\mathbb{C}[t,t^{-1}]$.
Let $a_1, \ldots,a_n \in\mathbb C^{\times}$ and
$\la_1, \ldots,\la_n\in\mathbb Z_+$ with $|\la|=\sum_i\la_i$. We define $B(a, \la)$ to be
the cyclic $\hat{\mathfrak{sl}}_2(\mathbb{C})-$module generated by $w$ such that
\begin{align*}
e(m)w&=f(0)^{|\la|+1}w=0, \qquad \forall m,\\
h(m)w&=\sum\limits_{j=1}^n \la_j a_j^mw, \qquad \forall m .
\end{align*}
The following result was proved by Chari and Pressley.
\begin{proposition}\cite{CP}\label{8} The $\hat{\mathfrak{sl}}_2(\mathbb{C})-$module
$B(a, \la)$ is finite dimensional. If $B^{'}$ is a finite dimensional
$\hat{\mathfrak{sl}}_2(\mathbb{C})$-module generated by
$w^{'}$ such that $\dim\,
U(\alpha^{\vee}\otimes\mathbb{C}[t,t^{-1}])w^{'}=1$, then $B^{'}$
is a quotient of some $B(a, \la)$ constructed above.
\end{proposition}

We also need the following remarkable formula
proved by Garland.
\begin{lemma}\cite{G} \label{14} Let
$\beta=\alpha+r_1\delta_1\in{\Delta_{\hat{\mathfrak g}_1+}}$, $r_1\in\mathbb{Z}$. Then for any $t\geq1$,
we have
\begin{equation*}
\begin{split}
x(\beta,\pm1)^t x(-\beta,0)^{t+1}=&\sum \limits_{m=0}^t x(-\beta,\pm m)\Lambda^{\pm}(\beta^{\vee},t-m)+X,\\
x(\beta,\pm1)^{t+1}x(-\beta,0)^{t+1}=&\Lambda^{\pm}(\beta^{\vee},t+1)+Y,
\end{split}
\end{equation*}
where X and Y are in the left ideal of $\mathfrak T$ generated by
the subalgebra $\mathfrak T_+$ and $\Lambda^{\pm}(\beta^{\vee},j)$ is the coefficient of $u^{j}$ in
$\Lambda^{\pm}(\beta^{\vee},u)=exp(-\sum\limits_{j=1}^{\infty}\frac{\beta^{\vee}t_2^{\pm
j} u^{j}}{j})$.
\end{lemma}

\begin{remark} By definition of
$\Lambda^{\pm}(\beta^{\vee},u)$ it follows that every element
$h\otimes t_2^{m}~(h\in\mathfrak{h}, m\in {\mathbb{Z}^{\times}})$
is a polynomial in the variables $\Lambda^{\pm}(\alpha_i^{\vee},j)~(i=1,2,\ldots,s; j\in\mathbb{N})$.
\end{remark}
\begin{theorem}\label{*}
For each $p=1, \ldots, s$, let $\lambda_{p,i}\in\mathbb Z_+,
 a_{p,i}\in\mathbb C^{\times}, i=1,\cdots, k_p$. If $\lambda$ satisfies
\begin{equation}\label{e:condition}
\lambda(\alpha_p^{\vee}(0,n))=\sum\limits_{i=1}^{k_p} \lambda_{p,i}
a_{p,i}^n
\end{equation}
and $\la(c_1)=\la(c_2)=0$, then
\begin{enumerate}
\item $L(\lambda)$ is integrable.
\item $L(\lambda)=U(\hat{\mathfrak
n}_{1-}\otimes\mathbb{C}[t_2]).v_{\lambda}$, where
$\hat{\mathfrak n}_{1-}$ is the negative nilpotent subalgebra of
$\hat{\mathfrak g}_1$.
\end{enumerate}
Moreover if $L(\la)$
and $L(\la^{'})$ are nonzero irreducible modules, then
\begin{equation}
L(\la)\cong L(\la^{'})\Longleftrightarrow \la=\la^{'}.
\end{equation} \end{theorem}
\begin{proof}
(1) By Lemma \ref{7} it suffices to show that $x(-\alpha_i,n)$ are
locally nilpotent on $L(\lambda)$ for $i=0,1,\ldots,s$ and
$n\in\mathbb{Z}$. Since $L(\lambda)$ is irreducible, one only needs to
prove that there exist $ N\geq 0$ such
that
\begin{equation}\label{9}
x(\alpha_j,m)x(-\alpha_i,n)^{N}.v_{\lambda}=0~(i,j=0,1,\ldots,s;~
m,n\in\mathbb{Z}).
\end{equation}

If $j\neq i$, then $\alpha_j-\alpha_i+(m+n)\delta_2$ is not a root.
So (\ref{9}) holds.

If $j=i$, denote $
x_n=x(\alpha_i,n), y_n=x(-\alpha_i,n),h_{n}=\alpha_i^{\vee}(0,n)$.
Then we have
\begin{equation*}
[x_m,y_n]=h_{m+n}, [h_p,x_m]=2x_{p+m}, [h_p,y_n]=-2y_{p+n}.
\end{equation*}
So $\{x_n,y_n,h_n: n\in\mathbb{Z}\}$ is a basis of the loop algebra
$\hat{\mathfrak{sl}}_2(\mathbb C)$. We consider the subspace
$U(\hat{\mathfrak{sl}}_2(\mathbb C))v_{\lambda}$ inside
$M(\lambda)$. It follows from Prop. \ref{8} that
$x(-\alpha_i,n)^{N}v_{\lambda}$ belongs to a proper submodule of
$U(\hat{\mathfrak{sl}}_2(\mathbb C))v_{\lambda}$ for some $N\geq 0$.
In fact, if $x(-\alpha_i,n)^{N}v_{\lambda}$ does not
belong to the proper maximal submodule $M$ of $U(\hat{\mathfrak{sl}}_2(\mathbb
C))v_{\lambda}$ for any $N\geq 0$,
then each $x(-\alpha_i,n)^{N}v_{\lambda}+M$ $(N\in\mathbb N)$ is non-zero in
the irreducible quotient $U^{'}=U(\hat{\mathfrak{sl}}_2(\mathbb C))v_{\lambda}/M$. Therefore
$U^{'}$ is infinite dimensional, but it is also isomorphic to some $B(a,\la)$. This
is a contradiction by Prop. \ref{8}. Applying PBW theorem to $M(\lambda)$, we get
Eq. (\ref{9}), which finishes the proof.

(2) We consider the action of real root vectors $x(-\alpha_i, -r)$,
where $i=0,1,\ldots,s$ and $r\in\mathbb Z_+$. By (1), there exists a
(minimal) positive integer $N_i$ such
that
\begin{equation}\label{15}
x(-\alpha_i,0)^{N_i+1}v_{\lambda}=0.
\end{equation}
 Let $x(\alpha_i, 1)^{N_i}$ act on Eq.(\ref{15}) and
by Lemma \ref{14}, we get
\begin{equation}\label{16}
\sum
\limits_{m=0}^{N_i}x(-\alpha_i,m)\Lambda^+(\alpha_i^{\vee},N_i-m).v_{\lambda}=0.
\end{equation}
Applying 
$\alpha_i^{\vee}\otimes t_2^{-r}$ to Eq.
(\ref{16}), we get
\begin{equation}\label{17}
\sum
\limits_{m=0}^{N_i}x(-\alpha_i,m-r)\Lambda^+(\alpha_i^{\vee},N_i-m).v_{\lambda}=0.
\end{equation}
Therefore $x(-\alpha_i,-r).v_{\lambda}$ is written as a linear
combination of the elements in the set
$\{x(-\alpha_i,m).v_{\lambda}, m>-r\}$.
We claim that $\Lambda^+(\alpha_i^{\vee},N_i).v_{\lambda}\neq 0$. In fact,
applying $\alpha_i^{\vee}(0,-1)$ to Eq.(\ref{15}), one gets that
$x(-\alpha_i,-1)x(-\alpha_i,0)^{N_i}v_{\lambda}=0$. Note that the
subalgebra $\{x(\alpha_i,1),x(-\alpha_i,-1),\alpha_i^{\vee}\}$ is
isomorphic to $sl_2$, then
$x(\alpha_i,1)^{q}x(-\alpha_i,0)^{N_i}v_{\lambda}\neq0~(0 \leq q
\leq N_i)$ by properties of $sl_2$-modules. When $N_i=0$,
$\Lambda^+(\alpha_i^{\vee},0)=1$. If $N_i=1$,
note that $x(\alpha_i,1)x(-\alpha_i,0)v_{\lambda}\neq0$, so
$\Lambda^+(\alpha_i^{\vee},1).v_{\lambda}=-\alpha_i^{\vee}(0,i).v_{\lambda}\neq0$.
If $N_i> 1$, we choose $q=N_i$ in the previous equation. Then Lemma \ref{14} implies that
$\Lambda^+(\alpha_i^{\vee},N_i).v_{\lambda}\neq 0$.
Using induction
on $r$, one shows that for arbitrary $r>0$, the element
$x(-\alpha_i,-r).v_{\lambda}$ can be represented by the elements of
the form $\{x(-\alpha_i,m).v_{\lambda}, m\geq 0\}$. This completes
the proof by Lemma \ref{L:gen1}. The last statement is easily seen.

\end{proof}

The IVM's have both finite and infinite dimensional weight
subspaces. In the following, we study the weight spaces
of an irreducible $L(\lambda)$ as a
module for
$\hat{\mathfrak{h}}'=\mathfrak{h}\oplus\mathbb{C}c_1\oplus\mathbb{C}c_2\oplus\mathbb{C}d_1$.

Let $\mathfrak
T'=\mathfrak{g}\otimes\mathbb{C}[t_1^{\pm1},t_2^{\pm1}]
\oplus\mathbb{C}c_1\oplus\mathbb{C}c_2\oplus\mathbb{C}d_1$. Set
$\hat{\mathfrak h}'[t_2^{\pm}]=\
span\{\mathfrak{h}(0,n),c_1,\\c_2,d_1; n\in\mathbb{Z}\}$, which is
abelian. Note that $\mathfrak T_{0}=\hat{\mathfrak h}'[t_2^{\pm}]
\oplus\mathbb{C}d_2$. Hence ${\mathfrak T'}=\mathfrak
T_+\oplus\mathfrak T_-\oplus\hat{\mathfrak h}'[t_2^{\pm}]$.

\begin{theorem}
If $\lambda$ satisfies the conditions of Theorem \ref{*} (cf. (\ref{e:condition})), then the
weight spaces of $L(\lambda)$ are finite dimensional as
$\mathfrak T'-$modules with respect to
$\hat{\mathfrak{h}}'=\mathfrak{h}\oplus\mathbb{C}c_1\oplus\mathbb{C}c_2\oplus\mathbb{C}d_1$.

\end{theorem}
\begin{proof}
Let
$\lambda|_{\hat{\mathfrak{h}}'}=\lambda_1$.
Since $d_2$ is removed, the root $\delta_2$ can be viewed as
nullified, thus the weight set $P(L(\lambda))\subset
\lambda_1-(\sum\limits_{i=0}^s \mathbb{Z}_+\alpha_i)$. Consider the
weight space $L(\lambda)_{\lambda_1-\varepsilon}$, where
$\varepsilon\in\sum\limits_{i=0}^s\mathbb{Z}_+\alpha_i$. By PBW
theorem, $L(\lambda)_{\lambda_1-\varepsilon}$ is spanned by
\begin{equation}\label{10}
x(\beta_1,n_1)x(\beta_2,n_2)\cdots x(\beta_k,n_k).v_{\lambda},
\end{equation}
where $\beta_1,\ldots,\beta_k$ are negative roots of the
affine Lie algebra $\hat{\mathfrak g}_1$ such that $\varepsilon=-\sum\limits_{i=0}^k
\beta_i$ and $n_i\in\mathbb{Z}$. There are only finitely many $\beta_i$
for a given $\varepsilon$.

For fixed $\beta_i$'s, $x(\beta_1,n_1)\cdots
x(\beta_k,n_k).v_{\lambda}\, (n_i\in\mathbb{Z})$ generate a finite
dimensional subspace. In fact, define
$e_p(t_2)=\prod\limits_{j=1}^{k_p}(t_2-a_{p,j})=\sum\limits_{i=0}^{k_p}\epsilon_{p,i}
t_2^i$. Let $E_p=e_p(t_2)\mathbb{C}[t_2,t_2^{-1}]$. By
$e_p(a_{pj})=0${\color{blue},} it is easy to check that
$$\lambda(\alpha_{p}^{\vee}\otimes E_p)=0,~p=1,2,\ldots,s.
$$
Let
$e(t_2)=\prod\limits_{p=1}^{s}e_p(t_2)=\sum\limits_{i=0}^{k}\epsilon_{i}
t_2^i$, where $k=\sum\limits_{i=1}^{p}k_i$. Set $E=e(t_2)\mathbb
C[t_2, t_2^{-1}]\subset E_p$, then
\begin{equation}\label{***}
\lambda(\alpha_{p}^{\vee}\otimes E)=0.
\end{equation}

First we show for any negative affine root $\beta$ of
$\hat{\mathfrak g}_1$
\begin{equation}\label{11}
\sum\limits_{i=0}^k\epsilon_ix(\beta,m+i).v_{\lambda}=0,\forall
m\in\mathbb{Z}
\end{equation}
in $L(\lambda)$. Since $L(\lambda)$ is
irreducible, it is enough to check that $\mathfrak T_+$ annihilates
the left-hand side of (\ref{11}).  By Lemma \ref{L:gen1} this means that
$x(\alpha_j, n)$ kills the LHS of (\ref{11}) for any $j\in
\{0,\ldots, s\}$. We use induction on $ht(-\beta)$.

First of all, let us consider the case of $ht(-\beta)=1$ say
$\beta=-\alpha_p$. If $p\neq j$, then clearly $x(\alpha_j, n)$
annihilates the LHS of (\ref{11}). Also
\begin{equation*}
x(\alpha_p,n)(\sum\limits_{i=0}^k\epsilon_i
x(-\alpha_p,m+i).v_{\lambda})
=\sum\limits_{i=0}^k\epsilon_i\alpha_p^{\vee}(0,m+n+i).v_{\lambda}=0
\end{equation*}
by (\ref{***}), where $\alpha_0^{\vee}=-\theta^{\vee}$ as $c_1$ acts
as $0$. Hence Eq. (\ref{11}) holds for $ht(-\beta)=1$.

Now consider general $\beta$ of $ht(-\beta)>1$. There exists a
simple root $\alpha_j$ such that $\alpha_j+\beta$ is a negative
affine root and $ht(-\alpha_j-\beta)<ht(-\beta)$. Therefore
\begin{align*}
x(\alpha_j,n)(\sum\limits_{i=0}^k\epsilon_ix(\beta,m+i).v_{\lambda})
=\sum\limits_{i=0}^k\epsilon_ix(\alpha_{j}+\beta,n+m+i).v_{\lambda}=0
\end{align*}
by the induction hypothesis.

Next we show that
\begin{equation}\label{12}
\sum\limits_{i=0}^k\epsilon_ix(\gamma_1,n_1)\cdots
x(\gamma_j,m+i)x(\gamma_{j+1},n_{j+1})\cdots
x(\gamma_l,n_l).v_{\lambda}=0
\end{equation}
for any fixed $\gamma_1,\cdots,\gamma_l$ in $\Delta_{\hat{\mathfrak g}_1-}$. 
This is again proved by another induction on $ht(-\gamma_{j+1}-\cdots-\gamma_l)$.

If the height of $-(\gamma_{j+1}+\cdots+\gamma_l)$ is 0, Eq.
(\ref{12}) is clear. Then
\begin{equation*}
\begin{split}
\sum\limits_{i=0}^k\epsilon_ix(\gamma_1,n_1)\cdots
x(\gamma_j,m+i)x(\gamma_{j+1},n_{j+1})\cdots x(\gamma_l,n_l).v_{\lambda}\\
=\sum\limits_{i=0}^k\epsilon_{i}x(\gamma_1,n_1)\cdots
[x(\gamma_j,m+i),x(\gamma_{j+1},n_{j+1})]\cdots
x(\gamma_l,n_l).v_{\lambda}+\\
\sum\limits_{i=0}^k\epsilon_ix(\gamma_1,n_1)\cdots
x(\gamma_{j+1},n_{j+1})x(\gamma_j,m+i)\cdots
x(\gamma_l,n_l).v_{\lambda}.
\end{split}
\end{equation*}
Each term of the right hand side is zero by induction hypothesis.
Therefore Eq.(\ref{12}) holds.

For fixed
$\beta_1,\beta_2,\ldots,\beta_k$, the vectors of the form (\ref{10})
generate a finite dimensional weight space due to the fact that
$\dim \mathbb C[t_2, t_2^{-1}]/E <\infty$. This finishes the proof.
\end{proof}

\begin{theorem}
If $L(\lambda)$ is irreducible as a ${\mathfrak T}'$-module with
finite dimensional weight spaces and the action of $c_1$ and $c_2$
are zero, then $\lambda$ satisfies the conditions of Theorem \ref{*}
(cf. (\ref{e:condition})).
\end{theorem}
\begin{proof} For $p\neq0$, the algebra $\mathfrak{L}$ generated by
$\{x_{\alpha_p},x_{-\alpha_p},\alpha_{p}^{\vee}\}$ is isomorphic to
$\mathfrak{sl}_2$.
 Let $V$ be the irreducible quotient
of $U(\hat{\mathfrak{L}}).v_{\la}$, where
$\hat{\mathfrak{L}}=\mathfrak{L}\otimes\mathbb C[t_2,
t_2^{-1}]\oplus\mathbb{C}c_2$. Since the weight spaces of
$L(\lambda)$ are finite dimensional, the
set $\{x(-\alpha_p,n).v_{\la}, n\in\mathbb{Z}\}$ is
linearly dependent. Thus there exists a nonzero
polynomial $g=\sum\limits_i g_i t_2^i$ such that
$x_{-\alpha_p}\otimes g.v_{\la}=0$. Let $G(t_2)=g\mathbb C[t_2,
t_2^{-1}]$, then $x_{-\alpha_p}\otimes G.v_{\la}=0$. In fact,
$$0=(\alpha_{p}^{\vee}\otimes t_2^m)x_{-\alpha_p}\otimes g.v_{\la}=(x_{-\alpha_p}\otimes g)
\alpha_{p}^{\vee}\otimes t_2^m.v_{\la}-2x_{-\alpha_p}\otimes t_2^m
g.v_{\la}$$ and $\alpha_{p}^{\vee}\otimes
t_2^m.v_{\la}=\la(\alpha_{p}^{\vee}\otimes t_2^m).v_{\la}$.
Naturally, $\alpha_{p}^{\vee}\otimes G.v_{\la}=0$. Subsequently $(\mathfrak{L}\otimes
G\oplus\mathbb{C}c_2).v_{\la}=0$.

Set $W=\{v\in V|(\mathfrak{L}\otimes
G\oplus\mathbb{C}c_2).v=0\}$. Clearly $W$ is a nonzero submodule of
$V$. Thus $V=W$ due to irreducibility of $V$. Then $V$ is a
$\hat{\mathfrak{L}}/\mathfrak{L}\otimes
G\oplus\mathbb{C}c_2-$module, thus $\dim V<\infty$. By Prop.
\ref{8}, $\lambda$ satisfies the conditions of Theorem \ref{*}.
\end{proof}

\begin{remark} The above proof also shows that when $c_1$ and $c_2$ act trivially,
the irreducible ${\mathfrak
T}'$-module $L(\lambda)$ has finite dimensional weight spaces if and
only if there is an ideal $\mathcal {S}$ of $\mathbb C[t_2,
t_2^{-1}]$ such that $\lambda(\alpha_p^{\vee}\otimes \mathcal {S})=
0$, $p=1,2,\ldots,s$.
\end{remark}

\begin{corollary}\label{13}
Let $L(\lambda)$ be an irreducible ${\mathfrak T'}$-module with
finite dimensional weight spaces and suppose that the centers $c_1$ and $c_2$
act trivially. Then there exists an ideal $\mathcal {S}$ of $\mathbb
C[t_2, t_2^{-1}]$ such that $\widetilde{\mathfrak{g}}_1\otimes
\mathcal {S}.L(\lambda)=0$, where
$\widetilde{\mathfrak{g}}_1=\mathfrak{g}\otimes\mathbb{C}[t_1,t_1^{-1}]$.

\end{corollary}
\begin{proof}
First there exists an ideal $\mathcal S$ of
$\mathbb{C}[t_2,t_2^{-1}]$ such that $\lambda(\alpha_i^{\vee}\otimes
\mathcal {S})= 0$ for all $i=1,2,\ldots,s$. By the definition of
$L(\lambda)$, we have $x_{\alpha_i}\otimes \mathcal
{S}_{.}v_{\lambda}= 0$ for each $i=0,1,2,\ldots,s$. The next step is
to show $y_{\alpha_{i}}\otimes \mathcal {S}_.v_{\lambda}= 0$ for
$i=0,1,2,\ldots,s$. Since $x_{\alpha_j}\otimes t_2^m.
y_{\alpha_i}\otimes \mathcal
{S}.v_{\lambda}=\delta_{ji}\alpha_i^{\vee}\otimes \mathcal
{S}.v_{\lambda}= 0$ for arbitrary $j,i=0,1,2,\ldots,s$,
$m\in\mathbb{Z}$~(where we set $\alpha_0^{\vee}=-\theta^{\vee}$),
and $L(\lambda)$ is irreducible, one sees that
$y_{\alpha_i}\otimes \mathcal {S}_.v_{\lambda}= 0$ for
$i=0,1,2,\ldots,s$. Hence $\widetilde{\mathfrak{g}}_1\otimes
\mathcal {S}.v_{\lambda}=0$ by induction.

Now consider $\overline{W}=\{w\in L(\lambda),\widetilde{\mathfrak{g}}_1
\otimes \mathcal {S}.w=0 \}$, which is a submodule of $L(\lambda)$.
Then $L(\lambda)=\overline{W}$ by the irreducibility of
$L(\lambda)$, therefore $\widetilde{\mathfrak{g}}_1\otimes \mathcal
{S}.L(\lambda)=0$.
\end{proof}

Since $\la(c_1)=\la(c_2)=0$, $L(\lambda)$ can be
viewed as a module for the loop algebra $\mathfrak{g}\otimes
\mathbb{C}[t_1,t_1^{-1},t_2,t_2^{-1}]$.
The following
proposition is easily derived from Corollary \ref{13} and \cite[Prop. 3.8]{RF}.

\begin{proposition} Let $\mathcal {S}_1$, $\mathcal {S}_2$ be co-prime and co-finite ideals
of $\mathbb C[t_2, t_2^{-1}]$, and suppose that $\lambda$ and $\mu$ satisfy the
conditions in Theorem \ref{*}. Then $L(\lambda+\mu)\cong
L(\lambda)\otimes L(\mu)$.
\end{proposition}

Rao \cite{SE1} and Chang-Tan \cite{CT}
have shown respectively that irreducible integrable modules for toroidal and TKK modules
with finite dimensional weight spaces and $c_1>0, c_2=0$ are highest weight
modules. In general our modules do not seem to be of highest weight type when $c_1=c_2=0$.

\bigskip

\centerline{\bf Acknowledgments} NJ
acknowledges the partial support of Simons Foundation grant 198129,
NSFC grant 11271138 and hospitality of Max-Planck Institute for Mathematics
in the Sciences at Leipzig during this work.

\bibliographystyle{amsalpha}

\end{document}